\documentclass[10pt,twocolumn]{IEEEtran}

\usepackage{graphicx,amsmath,amsfonts,amssymb,amsbsy,bm,amsthm,epstopdf} %

\newtheorem{lemma}{Lemma}

\def\cond{ {\scriptscriptstyle{\bullet}}  }

\def\e{ {\rm e}}

\def\bfA{ {\boldsymbol A} }
\def\bfB{ {\boldsymbol B} }
\def\bfC{ {\boldsymbol C} }

\def\bfF{ {\boldsymbol F} }
\def\bfG{ {\boldsymbol G} }
\def\bfI{ {\boldsymbol I} }
\def\bfJ{ {\boldsymbol J} }

\def\bfs{ {\boldsymbol s} }

\def\bfS{ {\boldsymbol S} }
\def\bfT{ {\boldsymbol T} }
\def\bfX{ {\boldsymbol X} }

\def\deltt{\Delta_{\rm t}}
\def\EqualDist{\mathrel{\mathop=^{\rm d}}}
\def\EqualDef{\mathrel{\mathop=^{\rm def}}}

\def\hbS{ {\hat{\bfS}} }

\newtheorem{remark}{Remark}

\def\eye{ {\rm i}}

\def\fN{f^{}_{\cal N}}

\def\EqualDist{\mathrel{\mathop=^{\rm d}}}
\def\DefEq{\mathrel{\mathop=^{\rm def}}}

\def\cov{\mathop{\rm cov}\nolimits}

\def\tr{ \mbox{tr} }

\def\fN{f^{}_{\cal N}}

\begin{document}


\title{Partial Coherence Estimation via Spectral Matrix Shrinkage under Quadratic Loss}

\author{D.~Schneider-Luftman, {\it Student Member, IEEE} and A.~T.~Walden, {\it Senior~Member, IEEE}\thanks{
Copyright (c) 2015 IEEE. Personal use of this material is permitted. However, permission to use this material for
any other purposes must be obtained from the IEEE by sending a request to pubs-permissions@ieee.org. 
Deborah Schneider-Luftman and Andrew Walden
are both at the Department of Mathematics, Imperial College  London, 180 Queen's Gate,
London SW7 2BZ, UK.  
(e-mail: deborah.schneider-luftman11@imperial.ac.uk and a.walden@imperial.ac.uk)
}}

\maketitle

\begin{abstract} 
Partial coherence is an important quantity derived from spectral or precision matrices and is used in seismology, meteorology, oceanography, neuroscience and elsewhere.
If the number of complex degrees of freedom only slightly exceeds the dimension of the multivariate stationary time series, spectral matrices are poorly conditioned and shrinkage techniques suggest themselves. When true partial coherencies are quite large then for shrinkage estimators of the diagonal weighting kind it is shown empirically that the minimization of risk using quadratic loss (QL) leads to oracle partial coherence estimators  superior to those derived by minimizing risk using Hilbert-Schmidt (HS) loss. When true partial coherencies are small the methods behave similarly. We derive two new QL estimators for spectral matrices, and new QL and HS estimators for precision matrices. In addition for the full estimation (non-oracle) case where certain trace expressions must also be estimated, we examine the behaviour of three different QL estimators, the precision matrix one seeming particularly robust and reliable. For the empirical study we  carry out exact simulations derived from real EEG data for two individuals, one having large, and the other  small, partial coherencies. This ensures our study covers cases of real-world relevance.

\end{abstract}



\section{Introduction}\label{sec:introduction}
Consider a  $p$-vector-valued (or multivariate) stationary time series $\{ {\bfX}_t
\}$ where
${\bfX}_t=[X_{1,t},\ldots,X_{p,t}]^T\in {\mathbb R}^p, \,t\in{\mathbb Z},$
and $^T$ denotes transposition. Without loss of generality we assume $\{\bfX_t\}$ to have a zero mean. Denote the sample interval by $\deltt.$ One very important quantity derived from vector-valued time series is the partial coherence between different series. 
Let $\bfS(f)$ denote the spectral matrix of $\{ \bfX_t\}$ at frequency $f,$
assumed to exist and be of full rank. With $\bfs_\tau {\displaystyle\EqualDef} \cov\{ \bfX_{t+\tau}, \bfX_t\}=
E\{ {\bfX}_{t+\tau} {\bfX}_t^T \},$ we have
$
\bfS(f){\displaystyle\EqualDef}\deltt \sum_{\tau=-\infty}^\infty {\bfs}_{\tau}
\,\e^{-{\eye}2\pi f \tau \,\deltt}. 
$
Denote the $(j,k)^{\rm th}$ element of the precision matrix $\bfC(f){\displaystyle\EqualDef}\bfS^{-1}(f)$ by $C_{jk}(f).$
The partial coherence between series $j$ and $k$ can be expressed as,
(e.g., \cite{Dahlhaus00}),
$
\gamma^2_{jk\cond\{\setminus jk\}}(f) 
= {|C_{jk}(f)|^2 / [C_{jj}(f) C_{kk}(f)] },
$
and is the frequency  domain squared correlation coefficient between series $j$ and $k$ {\it after the removal of the linear effects of the remaining series,} the remaining series being denoted by $\{\setminus jk\}.$ This characteristic has led to partial coherence being used widely in the physical sciences, e.g., in seismology \cite{White84}, meteorology \cite{Fraser77}, oceanography 
\cite{Koutitonsky_etal02} and extensively in neuroscience \cite{Fiecasetal10,Larsen_etal00,Mimaetal00,Salvador05,Medkouretal09}. Clearly to calculate the partial coherencies we can estimate $\bfS(f)$ as $\hbS(f),$ and invert it, or we can estimate $\bfC(f)$ directly.

Consider an  
estimator $\hbS(f),$ of $\bfS(f).$  An estimator may be computed by a multitaper scheme involving $K$ tapers (e.g., \cite{PercivalWalden93}).
(Throughout this paper we assume $K\ge p$ so that the spectral matrix is non-singular with probability one.)
The spectral matrices will be non-singular but poorly conditioned  if 
$K$ is only a little larger than $p$.  The derived partial coherencies will reflect this ill-conditioning. This study was motivated by exactly this problem in a neuroscience setting. Due to required low-pass filtering only a small frequency range was available for analysis and $K$ was necessarily kept small --- see Section~\ref{sec:EEGapp}. 
$K$ cannot be simply increased because of its connection to the implied smoothing bandwidth:  if $K$ is made larger, the required resolution may be lost. (Other estimators such as periodograms smoothed over frequencies have analogous properties.) 

Can the mean-square errors of the resulting estimated partial coherencies 
be reduced? Such a reduction would be very useful: for example, the estimated partial mutual information obtained from the  estimated partial coherencies are used in determining brain functional connectivity \cite{Salvador05} so that increased precision is scientifically well worthwhile.
An obvious approach is to use covariance shrinkage methodology, a subject with a large literature, see e.g., 
\cite{EfronMorris76,Haff80,Stein75,SchaferStrimmer05,FisherSun11,ChenWangMcKeown12}. 

Ledoit and Wolf (LW) \cite{LedoitWolf04} derived ideal shrinkage estimators which are a combination of the standard covariance matrix and a target matrix proportional to the identity; such diagonal up-weighting has a long history \cite{Carlson88} and we concentrate on this estimator class in this paper. 
LW minimize a risk measure --- defined in terms of Hilbert-Schmidt (HS)  or Frobenius loss --- between the true and estimated covariance matrices.  The ideal  parameter value for shrinking ${\hat\bfS}(f)$ is easily estimated. While convenient mathematically, Daniels and Kass \cite[p.~1174]{DanielsKass01} noted that use of the HS loss function can result in ``overshrinkage of the eigenvalues, especially the small ones.'' As we shall see, this warning is well-founded in terms of the estimation of partial coherencies: the LW estimator for shrinking ${\hat\bfS}(f)$ wipes out large partial coherencies.
In this paper we study alternatives to the LW-type estimator based on using the quadratic loss (QL) \cite{JamesStein61} rather than HS loss. 
QL was exploited \cite{Konno09} in shrinkage estimation for large dimensional covariance matrices but concentrating on singular estimators, with quite different estimators to those considered here.

We also look at estimating the precision matrix directly, rather than first inverting the estimated spectral matrix. \cite{WangPanTongZhu15} used
QL for precision matrix shrinkage in the context of large dimensional covariance matrices and derived estimators via random matrix theory.  Here, we address this problem without requiring large dimensionality, and without recourse to random matrix theory.
A weighted combination of the estimated inverse covariance matrix  and the identity was considered by \cite{EfronMorris76,Haff79}, and forms the basis for our approach, using both  the HS and QL losses. 

Our simulations are based on real 
electroencephalogram (EEG) time series data $(p=10),$  which in fact motivated this study. Both large and small partial coherencies are present and the data is perfect for demonstrating the behaviour  
of the various estimators. The simulations use the exact  circulant embedding methodology for vector-valued time series \cite{ChandnaWalden13}.

\subsection{Contributions}
Following some background on spectral estimation and the EEG data, the contributions of this paper are:

\begin{enumerate}
\item
We derive two  shrinkage estimators, QLa and QLb, for ${\hat \bfS}(f)$ under QL loss in closed form. One involves just a single shrinkage parameter, while the other has two parameters.
\item
The resulting oracle estimates of partial coherencies are compared to those from the LW scheme in terms of  the percentage relative improvement in squared error over that of the raw estimator. The LW scheme shrinks partial coherencies towards zero in a manner that renders it unreliable in practice when  any significant magnitude partial coherencies are present. The QL-based approaches perform much more robustly.
\item 
We next develop two-parameter  shrinkage estimators, HSP and QLP, for the spectral precision matrix ${\bfC}(f){\displaystyle\EqualDef}{\bfS}^{-1}(f)$ under both HS and QL loss, respectively, both in closed form.
The resulting oracle estimates of partial coherencies are compared to each other and to those of the other shrinkage estimators discussed above. 
\item
As a result
of 3. above, QLa, QLb and QLP estimators are further considered in the real-world --- non-oracle --- setting where the trace terms have also to be estimated and thence renamed QLa-est, QLb-est and QLP-est. It is found that  QLP-est
behaves in a very appealing robust way for both high and low true coherencies and is recommended as our preferred shrinkage method.
\end{enumerate}

Section~\ref{sec:backg} discusses the background spectral estimation, while 
Section~\ref{sec:EEGapp} outlines the necessary preprocessing of the 
EEG data and our simulation from it, showing the different partial coherence profiles for two individuals utilised in this work. 
Section~\ref{sec:conventional} introduces our two new QL-based shrinkage estimators for the spectral matrix. The standard HS-based shrinkage estimator is contrasted with the new QL-based estimators in Section~\ref{sec:egoracle} in terms of eigenvalue adjustment, shrinkage parameters and accuracy in estimating the partial coherencies.
New QL and HS estimators for precision matrices are derived in Section~\ref{sec:shrinkP} and the resultant partial coherence estimation is  analysed.
Since the QL approach has outperformed HS for the oracle estimators only the three estimators QLa, QLb and QLP are examined in Section~\ref{sec:nonoracle} where the full (non-oracle) estimators are examined, leading to QLP-est being our recommended approach.
Detailed derivations of all our new estimators are given in the Appendix.

\section{Spectral Matrix Estimation}  \label{sec:backg}
\subsection{Multitapering}

We make use of a set of $K\geq p$ orthonormal tapers $\{h_{k,t}\}, k=0,\ldots,K-1.$
A simple set are the sine tapers
\cite{RiedelSidorenko95}.
The elements of the $k^{\rm th}$ sinusoidal taper, are given by
$$ 
h_{k,t} = \left[\frac{2}{N+1}\right]^{1/2} \!\sin\,\left[ \frac{(k+1)\pi (t+1)}
 {N+1} \right],\,\, t=0,\ldots,N-1.
$$ 
For $t=0,\ldots,N-1,$
form the product
$h_{k,t} {\bfX}_t$ of the $t$th component of the $k$th taper with the $t$th component of the $p$-vector-valued
process, and for $k=0,\ldots,K-1$ compute the  vector Fourier transform
$
{\bfJ}_{k}(f)
{\displaystyle\EqualDef}
\deltt^{1/2}\sum_{t=0}^{N-1} h_{k,t} {\bfX}_t \,\e^{-\eye 2\pi ft\,\deltt}.
$
Let $\bfJ(f)$ be the $p\times K$ matrix defined by
$
\bfJ(f)=[\bfJ_0(f),\ldots,\bfJ_{K-1}(f)].
$
Then with ${\hbS}_k(f){\displaystyle{\EqualDef}}{\bfJ}_k(f) {\bfJ}^H_k(f)$the multitaper estimator of the $p\times p$ spectral matrix ${\bfS}(f)$ 
is
\begin{eqnarray}
{\hat {\bfS}} (f) &=&
\frac{1}{K} \sum_{k=0}^{K-1}  {\hbS}_k(f) =
\frac{1}{K} \bfJ(f)\bfJ^H(f).
\label{eq:eqnSHRdd}
\end{eqnarray}

Letting $B$ denote the bandwidth of the spectral window corresponding to the tapering, then
${\bfJ}_{k}(f), k=0,\ldots,K-1,$  
may be taken to be  independently and 
identically distributed 
as $p$-vector-valued complex Gaussian with mean zero and covariance matrix $\bfS(f):$
\begin{equation}
{\bfJ}_{k}(f)
 \EqualDist
{\cal N}_p^C\{{\bf 0},{\bfS}(f)\},
\label{eq:eqnSHRee}
\end{equation}
for $B/2< |f|< \fN-B/2$ for finite $N$ and Gaussian processes, or $0<|f|<\fN$ asymptotically \cite{ChandnaWalden11}. 
The bandwidth $B$ for sine tapers is given by
$
B=(K+1)/[(N+1)\deltt],
$
(e.g., \cite{Waldenetal95}). The choice of $K$ is therefore linked to the bandwidth chosen for the spectral window.
It is assumed to have been chosen narrow enough 
to ensure the components of ${\bfS}(f)$ are essentially constant across it. 
${\hat\bfS}(f)$ in  (\ref{eq:eqnSHRdd}) is the maximum-likelihood estimator for $\bfS(f),$ \cite{Goodman63}.

Given (\ref{eq:eqnSHRee}),  and with $K\geq p,$
$K\hbS(f)$ has the complex Wishart distribution with mean $K\bfS(f),$
written as
\begin{equation}
K\hbS(f)
 \EqualDist
{\cal W}_p^C\{K,{\bfS}(f)\}.
\label{eq:eqnWish}
\end{equation}
Therefore,
\begin{equation}
E\{ {\hat {\bfS}} (f)\}\!=\bfS(f) \quad \mbox{and}\quad E\{ \tr\{ {\hat{\bfS}}\} \}=\tr\{\bfS\},
\label{eq:eqnSHRll}
\end{equation}
two simple results which will be made use of. Other properties following from (\ref{eq:eqnWish}) will be introduced where appropriate.

\section{Application to EEG data}\label{sec:EEGapp}
We shall illustrate the methodology via data simulated from real 
electroencephalogram (EEG) data, (resting conditions with eyes closed),  \cite{Medkouretal10}.

\subsection{Preprocessing}
Real EEG signals were recorded on the scalp  at $10$  sites, 
using a bandpass filter of 0.5--45Hz and sample interval of $\deltt=0.01$s. 
The recorded data $\{\bfX_t\}$ is thus a $p=10$ vector-valued process, with Nyquist frequency 50Hz.
To fully remove any influence of the highly dominant and contaminating 10Hz alpha rhythm, which would otherwise cause severe spectral leakage,  a 4.6Hz low-pass filter was applied,  
followed by resampling to a sample interval of $\deltt=0.05$s giving a new Nyquist frequency of $\fN=10$Hz. 
After this downsampling each channel of data had $N=1024$ time series values.
The interesting EEG delta frequency range, $0.5< f \leq 4$ Hz \cite{Medkouretal10} should be reliably represented after this preprocessing.


\subsection{Simulation Strategy}\label{subsec:simulation}
From this 10-channel series  the spectral matrix $\bfS(f)$ was estimated as ${\bfS}_0(f),$ say, for $|f|\leq \fN.$ 
Using the vector-valued circulant embedding approach, 
\cite{ChandnaWalden13},
a large number, $M$ say, of independent Gaussian 10-channel time series were computed, each having  $\bfS_0(f),\, |f|\leq \fN,$ as its true spectral matrix. 
For each of these time series the estimated spectral matrix $\hbS(f)$  was computed using multitaper estimation with some specified number $K$ of sine tapers. These $M$ independent estimates of $\bfS_0(f)$ can then be used  to deduce various sampling properties for quantities derived from ${\hat\bfS}(f).$ As mentioned in the Introduction, we will look at cases where $K$ just exceeds $p,$ so that ill-conditioning is present. 

There are $p(p-1)/2=45$ partial coherencies as a function of frequency in this case. Fig.~\ref{fig:samplepcs} shows
the first 9 partial coherency plots ($j=1,k=2:10)$ for  individual 1 calculated from  $\bfS_0(f).$ These will be referred to as the true partial coherencies (that we are trying to estimate from the simulated data). We can see that there is a good range of values from nearly zero up to around 0.8. We shall also look at true partial coherencies for individual 2 who has most partial coherencies close to zero (the `sparse' case) --- see Fig.~\ref{fig:samplepcs6}. The two cases are summarized in Fig.~\ref{fig:avpcoh} which shows that for individual one there are `spikes' of high true partial coherence  around 2 and 3.25 Hz. These will be significant for our partial coherence estimators.

Unless stated otherwise, results  apply to individual 1.


\begin{figure}[t]
\centering
\includegraphics{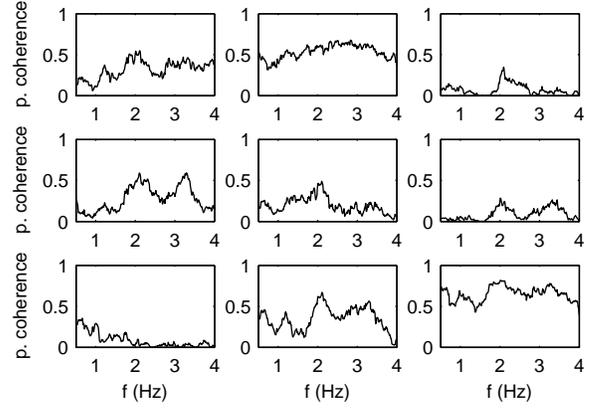}
\caption{The first 9 partial coherency plots ($j=1,k=2:10)$ for  individual one, found from the known matrix $\bfS_0(f).$
}
\label{fig:samplepcs}
\end{figure}
\begin{figure}[t]
\centering
\includegraphics{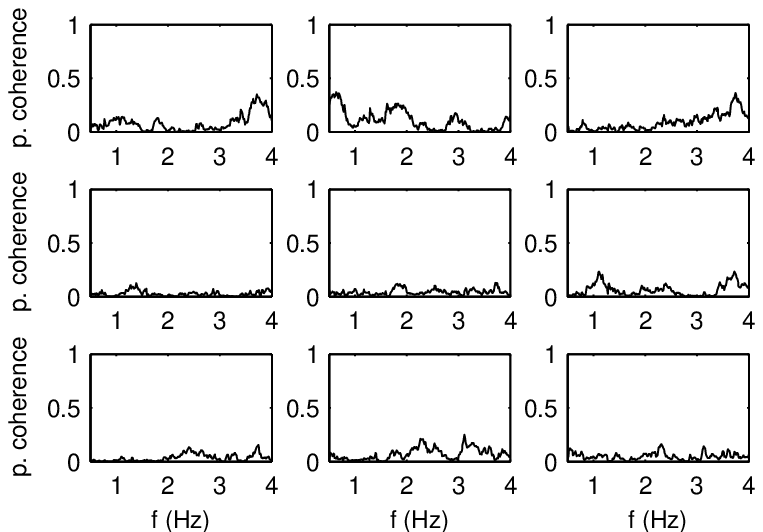}
\caption{The first 9 partial coherency plots ($j=1,k=2:10)$ for  individual two, found from the known matrix $\bfS_0(f).$
}
\label{fig:samplepcs6}
\end{figure}

\begin{figure}[t]
\centering
$
\begin{matrix}
\includegraphics{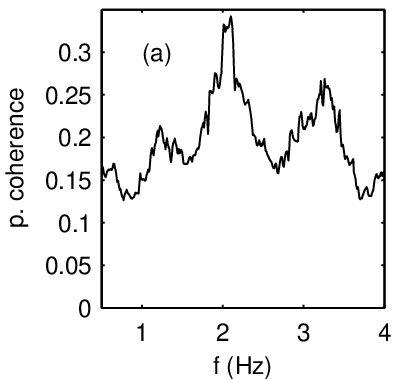}
\includegraphics{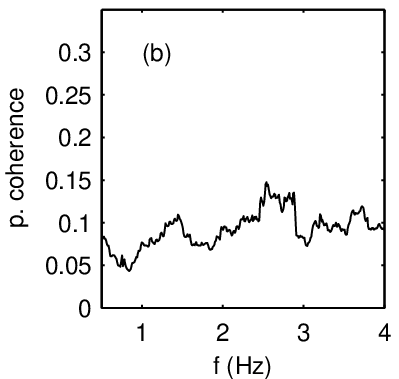}\\
\end{matrix}
$
\caption{Averaged true partial coherencies for (a) individual one, and (b) individual two, (averaging over 
all pairs $(j,k), 1\leq j<k\leq p$).
}
\label{fig:avpcoh}
\end{figure}

\section{Shrinking the Spectral Matrix}\label{sec:conventional}
\subsection{Conventional Approach}
The conventional
approach to `covariance matrix' regularization which has been extensively studied involves the forming of a convex combination of the sample covariance matrix and some well-conditioned `target' matrix. For an estimated $p\times p$ Hermitian spectral matrix ${\hat{\bfS}}(f)$ this would take the form
$
{\bfS}^{\star}(\rho(f))=(1-\rho(f)) {\hat{\bfS}}(f)+\rho(f){{\bfT}}(f),
$
where $\rho(f)\in(0,1)$ is known as the shrinkage parameter and ${\hat{\bfT}}(f)$ is the target matrix. Provided 
${\hat {\bfS}}(f)$ and ${\hat{\bfT}}(f)$ are both positive definite, then this convex combination   will itself be positive definite.

For notational brevity we shall drop the explicit frequency dependence in most of what follows.

The shrinkage coefficient $\rho$ is set such that it minimizes a risk criterion for some given loss ${\mathcal L},$ say, of $ \hat{\bfS}^{\star}(\rho)$ to the true matrix $\bfS:$ 
\begin{equation}\label{eq:basicShrinkage}
\rho_0 = \underset{\rho \in (0,1)}{\arg\min}\,  {\mathcal R}_{\mathcal L}(\hat{\bfS}^{\star}(\rho),\bfS).
\end{equation}
\subsection{Hilbert-Schmidt loss}
A common choice is the Hilbert-Schmidt  (HS) loss 
where, with ${\cal L}=HS$
\begin{eqnarray}
{\mathcal R}_{HS}(\hat{\bfS}^{\star}(\rho),\bfS) &{\displaystyle\DefEq} & E\{\tr\{(\hat{\bfS}^{\star}(\rho)-\bfS)^2\}\}\nonumber\\
&=& E\{ || {\bfS}^{\star}(\rho) -\bfS ||^2_{\rm F}\},\label{eq:riskHS}
\end{eqnarray}
where, for $\bfA\in{\mathbb C}^{p\times p}$, $|| \bfA ||_{\rm F}$ denotes the Frobenius norm
$
||\bfA||_{\rm F}=[\tr\{\bfA \bfA^H\}]^{1/2}
,$ $\tr\{\cdot\}$ denotes trace,
and $^H$ denotes complex-conjugate (Hermitian) transpose.

Using (\ref{eq:basicShrinkage}) and (\ref{eq:riskHS}) the target  matrix was chosen to be  of the form  $\bfT=(\tr\{ \bfS \}/p)\bfI_p$ in  \cite{LedoitWolf04} (for  real-valued covariance matrices) so that
\begin{equation}
{\bfS}^{\star}(\rho)=(1-\rho) {\hat{\bfS}}+\rho\frac{\tr\{\bfS\}}{p}\bfI_p.
\label{eq:eqnSHRff}
\end{equation}
For the ill-conditioned case ($p < K, p \simeq K$)  the estimator diagonally loads  the initial matrix $\hat{\bfS}$ and increases the zero or near-zero eigenvalues. 
For this choice  the solution to (\ref{eq:basicShrinkage}) and (\ref{eq:riskHS}) is 
(e.g.
\cite[eq.~9]{WaldenSLuftman15}), 
\begin{equation}\label{eq:HSrhoa}
\rho_0
=\left[1-\frac{K}{p}+
K\frac{\tr\{\bfS^2\}}{\tr^2\{\bfS\}}\right]^{-1}.
\end{equation}

\subsection{Quadratic Loss}

One alternative choice of loss criterion is the quadratic loss function:
\begin{equation}
{\mathcal R}_{QL}(\hat{\bfS}^\star(\rho), \bfS) = E\{\tr\{({\hat\bfS}^\star(\rho)\bfS^{-1} - \bfI_p)^2\}\}.\label{eq:riskQL}
\end{equation}
\begin{lemma}\label{lemma:QL2}
With ${\bfS}^{\star}(\rho)$ defined as in (\ref{eq:eqnSHRff})   the form of $\rho_0$ solving (\ref{eq:basicShrinkage}) and (\ref{eq:riskQL}) is
\begin{equation}\label{eq:rhoQLDet}
\frac{p^2}{(Kp+p^2)-\frac{2}{p^2}\left[Kp\,\tr\{\bfS^{-1}\}\tr\{\bfS\}-\frac{K}{2} \tr^2\{\bfS\}\tr\{\bfS^{-2}\}\right]}.
\end{equation}
\end{lemma}
\begin{proof}
This is given in Appendix\ref{app:rhoQLDet}.
\end{proof}
For easy identification we shall call this the QLa method.
\begin{figure}[t]
\centering
\includegraphics{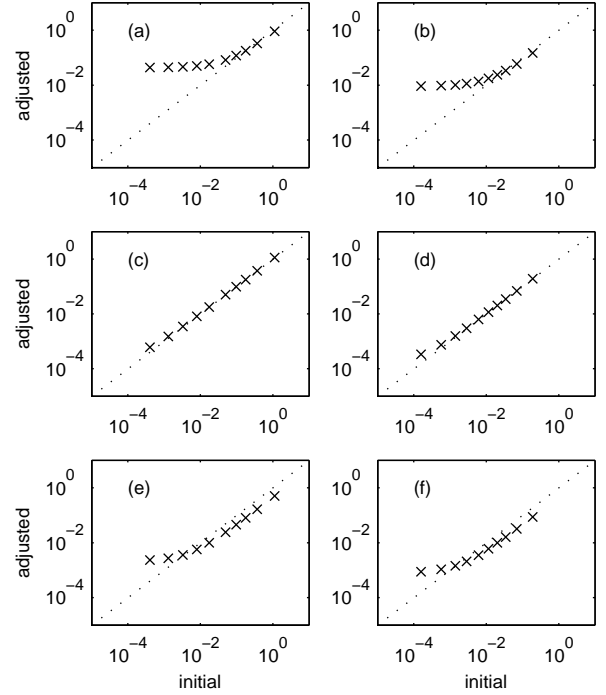}
\caption{
Average eigenvalue adjustment
for frequencies $f=0.85$Hz (left column) and 
$3.85$Hz (right column). The first row is for the
HS method,  the second for QLa and the third for QLb. 
 $K=12$ here.}
\label{fig:eigenvalues}
\end{figure}
\begin{figure}[t]
\centering
\includegraphics{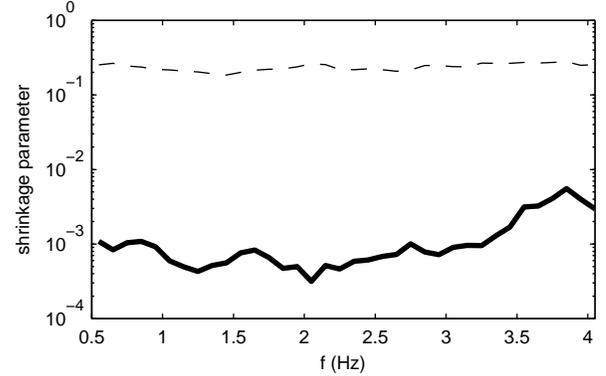}
\caption{Oracle value of $\rho(f)$ for HS (dashed line) and QLa (thick line)
for frequencies 0.55 to 4.05 in steps of 0.1.
}
\label{fig:shrinkparam}
\end{figure}

\subsection{Generalization}
\label{sec:QLshrinkage}
The shrinkage concept  can be extended to shrinkage estimators with generalised scaled identity targets, 
\begin{equation}\label{eq:genshrink}
{\hat\bfS}^{\star}(\eta,\rho) = (1-\rho){\hat\bfS}+\rho\, \eta\, \bfI_p ,\quad \eta >0,\rho \in (0,1) 
\end{equation}
with the optimum  shrinkage $(\rho_0)$ and scaling   parameter values $(\eta_0)$ given by
\begin{equation}
(\eta_0,\rho_0 )=  \underset{\eta>0 ,\,\rho \in (0,1) }{\arg\min} {\mathcal R}_{\mathcal{L}}(\hat{\bfS}^{\star}(\eta,\rho),\bfS),
\label{eq:LWestimator}
\end{equation}
for any convex loss criterion ${\mathcal L}.$
\begin{lemma}\label{lemma:twopar}
The solutions of  (\ref{eq:LWestimator}) for both the HS and QL losses are
\begin{eqnarray}
(\eta_0,\rho_0)_{\rm HS}=\left( \frac{\tr\{\bfS\}}{p}, \left[1-\frac{K}{p}+
K\frac{\tr\{\bfS^2\}}{\tr^2\{\bfS\}}\right]^{-1}\right),\label{eq:theorHS}\\
(\eta_0,\rho_0)_{\rm QL}=\left(\frac{\tr\{\bfS^{-1}\}}{\tr\{\bfS^{-2}\}},\left[ 1+ \frac{K}{p}- 
\frac{K\tr^2\{\bfS^{-1}\}}{p^2\tr\{\bfS^{-2}\}}\right]^{-1}\right).\label{eq:theorQL}
\end{eqnarray}
\end{lemma}
\begin{proof}
This is given in Appendix \ref{app:twoparamorig}.
\end{proof}
\begin{remark}
We see that the  choice of target 
$\bfT=(\tr\{ \bfS \}/p)\bfI_p$ is identical to $\eta_0$ under the HS loss, i.e., the scaling factor 
$(\tr\{ \bfS \}/p)$ 
is an optimal choice under the HS loss for shrinkage model (\ref{eq:genshrink}). Since the targets are then the same the form of $\rho_0$ was already known from (\ref{eq:HSrhoa}). 
\end{remark}

We shall call the method defined by
(\ref{eq:theorHS}) the HS method. We also call the method defined by (\ref{eq:theorQL}) the QLb method.

In deriving the optimal forms of $\eta_0$ and $\rho_0$
(Appendix\ref{app:twoparamorig}), we note that the form of $\eta_0$ for both losses arises independently of $\rho_0$. This leads to the following result.
\begin{lemma}\label{lemma:reparam}
If we reparameterize the shrinkage model in
(\ref{eq:genshrink}) to
$
{\hat\bfS}^{\star}(\alpha,\beta) = \alpha\,{\hat\bfS}+\beta\, \bfI_p ,\quad \alpha,\beta >0,
$
it follows that $(\alpha_0,\beta_0)$ defined by
$
(\alpha_0,\beta_0 )=  \underset{\alpha, \beta>0 }{\arg\min} {\mathcal R}_{\mathcal{L}}(\hat{\bfS}^{\star}(\alpha,\beta),\bfS),
$
are given by
$
(\alpha_0,\beta_0)_{\text{HS or QL}}=((1-\rho_0),\rho_0\eta_0)_{\text{HS or QL}}.
$
\end{lemma}
\begin{proof}
See Appendix \ref{app:twoparamnew}.
\end{proof}

\section{Example Oracle Behaviour}\label{sec:egoracle}
The above oracle solutions require exact knowledge of the true spectral matrix. Before looking at estimation methods, we  examine some aspects of the oracle solutions.

\subsection{Transformation of Eigenvalues}
Minimizing  the HS and QL risks effectively transforms the raw eigenvalues of ${\hat\bfS}$. Fig.~\ref{fig:eigenvalues} shows the transformation of the $p=10$ eigenvalues for HS, QLa and QLb at frequencies 0.8Hz and 3.85Hz. Both the initial and adjusted eigenvalues shown are averages, following sorting, over the $M=500$ simulations of ${\hat\bfS}.$ The main feature, which is persistent at other frequencies and data sets, is that the HS method leads to a large increase in the small eigenvalues and a very small decrease in the larger eigenvalues. The QLa method increases only the smallest eigenvalues by a small amount, and barely changes the large ones. The QLb method increases the small eigenvalues more than QLa but much less so than HS and decreaseses the larger eigenvalues by a larger amount than HS.
Such behaviour will have a significant effect in the estimation of partial coherencies.

\subsection{Shrinkage Parameters}
For methods HS and QLa the shrinkage parameter $\rho$ may be compared directly as the model  is the same in both cases. Fig.~\ref{fig:shrinkparam} shows the the oracle shrinkage parameters. From Fig.~\ref{fig:eigenvalues}(a),(b) and (c),(d) we are not surprised to see that the shrinkage parameter is orders of magnitude larger for HS than for QLa. For QLa it varies between around $10^{-3}$ to $10^{-2}$ while for HS it is around 0.25.

\subsection{Partial coherence}\label{subsec:pcohsim}
The 45 partial coherencies at each frequency can be quite variable in size from near zero to near unity.  Let ${\hat\gamma}^2_{jk\cond\{\setminus jk\}}(m;f_l)$ denote the estimate -- by any method -- of the partial coherence at frequency $f_l$ and replication number $m$ where $m=1,\ldots,M.$ To measure the quality of the estimates of partial coherencies  we firstly calculated the sum of squared errors
over all pairs $(j,k), 1\leq j<k\leq p:$
$$
{\cal E}(m; f_l)=\sum_{j,k} \left[{\hat\gamma}^2_{jk\cond\{\setminus jk\}}(m;f_l)-{\gamma}^2_{jk\cond\{\setminus jk\}}(f_l)\right]^2.
$$
We then averaged these over replications to get
$ {\bar {\cal E}}(f_l)=(1/M)\sum_{m=1}^M {\cal E}(m; f_l).$
Finally a percentage relative improvement in squared error (PRISE) was computed as
$$
{\rm PRISE}(f_l)=100 \left[\frac{{\bar {\cal E}}_B(f_l)-{\bar {\cal E}}_M(f_l)}{{\bar {\cal E}}_B(f_l)}\right]\%,
$$
where ${\bar {\cal E}}_B(f)$ denotes the average summed squared errors of the raw or basic estimate, and ${\bar {\cal E}}_M(f)$ denotes the same for any of the HS, QLa or QLb shrinkage methods.

\begin{figure}[t]
\centering
$
\begin{matrix}
\includegraphics{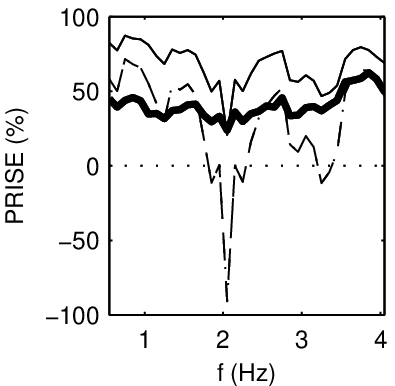}
\includegraphics{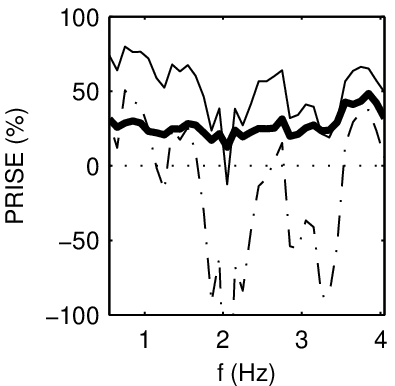}\\
\includegraphics{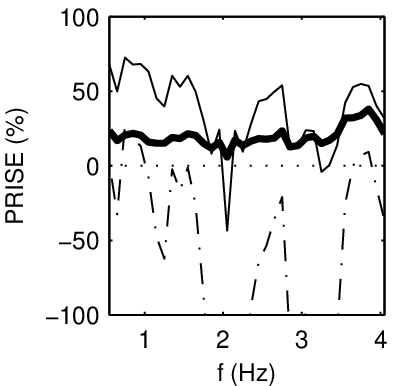}
\includegraphics{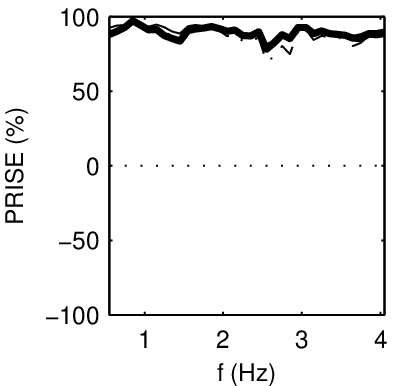}
\end{matrix}
$
\caption{Oracle PRISE with partial coherencies
derived from spectral matrices.  HS (dash-dot line), QLa (thick line) and QLb (thin line)
for $K=12$ (top left), $K=14$ (top right),
$K=16$ (bottom left), all for individual one, and
$K=12$ (bottom right) for individual two.
}
\label{fig:PRISE_SPEC}
\end{figure}
We are thus able to determine the improvement  in using any of the three shrinkage methods over using just the raw estimates. Results are shown in Fig.~\ref{fig:PRISE_SPEC} with results at frequencies  $0.55$ to $4.05$ in steps of 0.1. These are oracle results: the known quantities $\tr\{\bfS_0(f)\}, \tr\{\bfS^2_0(f)\},
\tr\{\bfS^{-1}_0(f)\}, \tr\{\bfS^{-2}_0(f)\}$  have been used
in the various estimators.

We also averaged the PRISEs over the frequencies to obtain Table~\ref{tab:avPRISE}.
Looking at the top three rows for individual one, we see that
the HS method does very poorly, being worse than using the raw estimates. QLb generally does well but is very variable (Fig.~\ref{fig:PRISE_SPEC}) and rather unpredictable with frequency. QLa seems to behave very nicely giving a fairly frequency-constant improvement  over using the raw estimates. In the sparse non-null partial coherence case of patient two, all methods do well.
\begin{table}[t]
\begin{center}
\begin{tabular}{|c|c|c|c|c|}
\hline 
\multicolumn{1}{|c|}{method} & \multicolumn{3}{|c|}{individual one} &
\multicolumn{1}{|c|}{individual two}
\tabularnewline
 & $K=12$ & $K=14$ & $K=16$ & $K=12$\tabularnewline
\hline \hline 
HS & 31 & -17 & -66 & 87 \tabularnewline
\hline 
QLa & 41 & 27  & 20 & 89 \tabularnewline
\hline 
QLb & 68 & 51  & 37 & 91 \tabularnewline
\hline\hline 
HSP & 27 & 7  & -7 & 82 \tabularnewline
\hline 
QLP & 52 & 32  & 23 & 85 \tabularnewline
\hline\hline
QLa-est & 26 & 18  & 13 & 66 \tabularnewline
\hline 
QLb-est & 65 & 49  & 34 & 87 \tabularnewline
\hline 
QLP-est & 52 & 34  & 25 & 84 \tabularnewline
\hline 
\end{tabular}
\end{center}
\caption{Average PRISE($\%$) over frequencies. }
\label{tab:avPRISE}
\end{table}
\section{Shrinking the Precision Matrix}\label{sec:shrinkP}
Since the partial coherence is derived from the precision matrix $\bfC(f)=\bfS^{-1}(f)$ we can also consider  shrinkage for this matrix. A possible approach is to take
\begin{equation}\label{eq:genshrinkc}
{\hat\bfC}^{\star}(\alpha,\beta) = \alpha\,{\hat\bfS}^{-1}+\beta\, \bfI_p ,\quad \alpha,\beta >0.
\end{equation}
Such a model appears in \cite{EfronMorris76} and \cite{Haff79} with the slight modification that the right side takes the form 
\begin{equation}\label{eq:Haff}
\alpha\,{\hat\bfS}^{-1}+ {{\beta}/{\tr\{ {\hat\bfS}\} }}\bfI_p.
\end{equation}
\begin{lemma}\label{lemma:precisionHS}
Under the model (\ref{eq:genshrinkc}), the HS risk
$
{\mathcal R}_{HS}(\hat{\bfC}^\star(\alpha,\beta), \bfC) = E\{\tr\{({\hat\bfC}^\star(\alpha,\beta)-\bfS^{-1} )^2\}\}
$
is minimized, for $K>p+1,$ by
\begin{eqnarray}
\alpha_0 &=&\frac{1}{D}\left[
p\frac
{ \tr\{ \bfS^{-2} \} }{\tr^2\{\bfS^{-1}\} } -1
\right]\label{eq:PalpHS}\\
\beta_0 &=& \frac{c_3 \,\tr\{\bfS^{-1}\} }{D}\left[ \frac
{ \tr\{ \bfS^{-2} \} }{\tr^2\{\bfS^{-1}\} }+(K-p)\right],\nonumber
\end{eqnarray}
where 
\begin{eqnarray*}
c_3 &=& \frac{K }{(K-p)^3-(K-p)}\\
D&=& \left[c_3 p (K-p)^2\frac
{ \tr\{ \bfS^{-2} \} }{\tr^2\{\bfS^{-1}\} }  +c_3 p(K-p)-\frac{K}{(K-p)} \right] .
\end{eqnarray*}
\end{lemma}
\begin{proof}
See Appendix \ref{app:PrecisHS}. 
\end{proof}
We shall call this the HSP method.

\begin{lemma}\label{lemma:precisionQL}
Under the model (\ref{eq:genshrinkc}), the QL risk
$
{\mathcal R}_{QL}(\hat{\bfC}^\star(\alpha,\beta), \bfC) = E\{\tr\{({\hat\bfC}^\star(\alpha,\beta)\bfS - \bfI_p)^2\}\}
$
is minimized, for $K>p+1,$ by
\begin{equation}\label{eq:PalpQL}
\alpha_0=\frac{1}{D}\left[
p\frac
{ \tr\{ \bfS^2 \} }{\tr^2\{\bfS\} } -1
\right]; \,\,
\beta_0=\frac{p}{D\,\tr\{ \bfS \}}{ \left[c_0-  \frac{K}{K-p}\right] }
\end{equation}
where 
\begin{equation}\label{eq:defC0}
c_0=\frac{K^2 }{(K-p)^2-1};\,\,D=\left[c_0 p \frac
{ \tr\{ \bfS^2 \} }{\tr^2\{\bfS\} } - \frac{K}{K-p} \right] .
\end{equation}
\end{lemma}
\begin{proof}
See Appendix \ref{app:PrecisQL}. 
\end{proof}
We shall call this the QLP method.
\begin{remark}  
We note that the numerator of $\alpha_0$ in (\ref{eq:PalpHS}) is essentially the $U$-statistic for testing for sphericity of $\bfC$
and in (\ref{eq:PalpQL}) for testing for sphericity of $\bfS.$ 
\end{remark}

Repeating the simulations in Section~\ref{subsec:pcohsim}\/ using the oracle estimates ${\hat\bfC}^{\star}(\alpha_0,\beta_0)$ for the precision matrices, and hence the partial coherencies, gave the results in Fig.~\ref{fig:PRISE_P} and lines 4 and 5 of Table~\ref{tab:avPRISE}. 
We see that HSP and HS are comparable for patient one for $K=12$ but for 
$K=14$ and $16,$ HSP does better. For individual two HSP is slightly worse than HS when $K=12.$ QLP performs much better than HSP, and for individual one 
performs intermediate to QLa and QLb.

\begin{figure}[t]
\centering
$
\begin{matrix}
\includegraphics{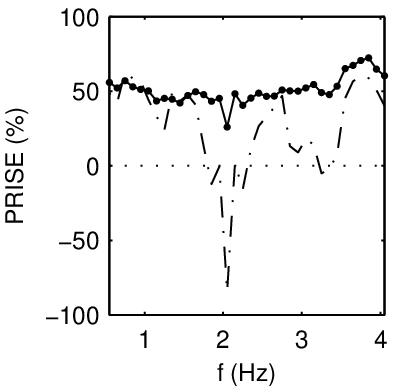}
\includegraphics{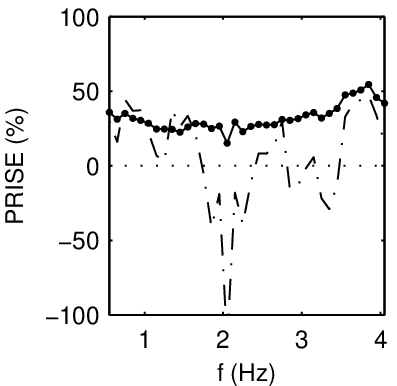}\\
\includegraphics{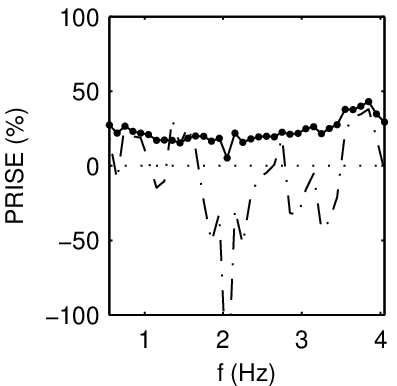}
\includegraphics{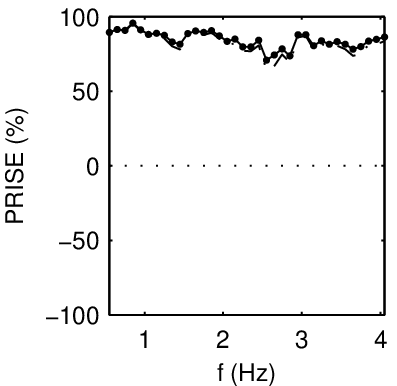}
\end{matrix}
$
\caption{Oracle PRISE with partial coherencies
derived from precision matrices.  HSP (dash-dot line) and QLP (line with bold dots)
for $K=12$ (top left), $K=14$ (top right),
$K=16$ (bottom left), all for individual one, and
$K=12$ (bottom right) for individual two.
}
\label{fig:PRISE_P}
\end{figure}

\section{Full (Non-Oracle) Estimation}\label{sec:nonoracle}
Each of the  estimators we have derived involve the traces of some subset of $\bfS, \bfS^2, \bfS^{-1},\bfS^{-2}.$ For the oracle estimators these were taken as known, but now we turn to the real-world case where these must also  be estimated.
Since the QL approach has performed uniformly better than HS for the oracle estimators we now only examine the three estimators QLa, QLb and QLP. 

Now QLP involves only $\tr\{ \bfS \}$ and $\tr\{ \bfS^2 \}$ whereas QLa and QLb involve   $\tr\{ \bfS^{-1} \}$ and $\tr\{ \bfS^{-2} \},$ which we would expect to be more problematic to estimate.

We recall that ${\hat\bfS}$ is the maximum likelihood estimator (MLE) for $\bfS.$ By the invariance property of MLEs \cite[p.~294]{CasellaBerger90} it follows that the MLE of $\tr\{ \bfS \}$ is $\tr\{ {\hat\bfS} \},$ of $\tr\{ {\bfS}^2\}$ is 
$\tr\{ {\hat\bfS}^2\},$ of $\tr\{ {\bfS}^{-1}\}$ is $\tr\{ {\hat\bfS}^{-1}\}$ and of $\tr\{ {\bfS}^{-2}\}$ is $ \tr\{ {\hat\bfS}^{-2}\}.$ The MLE is a strongly consistent estimator \cite[p.~129]{YoungSmith05}.
\begin{figure}[t]
\centering
$
\begin{matrix}
\includegraphics{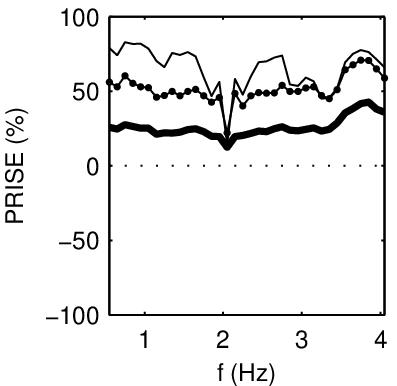}
\includegraphics{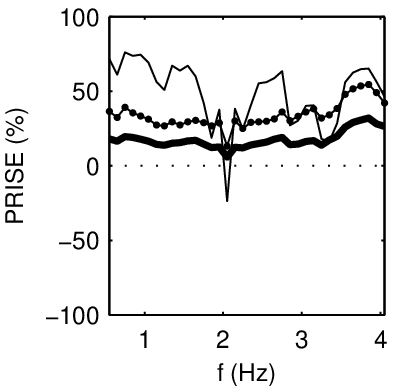}\\
\includegraphics{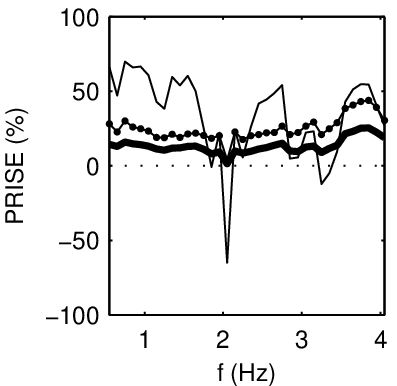}
\includegraphics{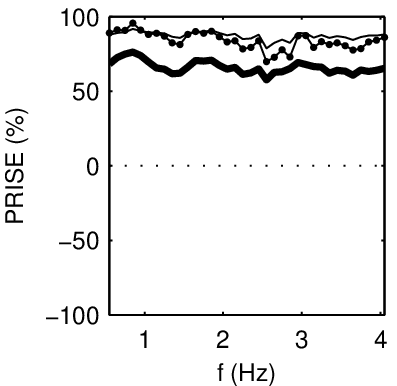}
\end{matrix}
$
\caption{PRISE with partial coherencies
derived from fully estimated spectral or precision matrices.  QLa-est (thick line), QLb-est (thin line) and QLP-est (line with bold dots)
for $K=12$ (top left), $K=14$ (top right),
$K=16$ (bottom left), all for individual one, and
$K=12$ (bottom right) for individual two.
}
\label{fig:PRISE_EST}
\end{figure}

\begin{itemize}
\item
To estimate $\tr\{ {\bfS}\}$ we use $\tr\{ {\hat\bfS}\}$ which by 
(\ref{eq:eqnSHRll})  is exactly unbiased.
\item
Our estimator for $\tr\{ {\bfS}^2\}$ is $\tr\{ {\hat\bfS}^2\}-(1/K) \tr^2\{ {\hat\bfS}\}.$
Using (\ref{eq:firstmom}) and (\ref{eq:secondmom}) we see that 
$$
E\left\{ \tr\{ {\hat\bfS}^2\}-\frac{1}{K} \tr^2\{ {\hat\bfS}\}  \right\}=\left[1-\frac{1}{K^2}\right]\tr\{ {\bfS}^2\},
$$
so the estimator  is asymptotically  unbiased for  
$\tr\{ {\bfS}^2\}.$
\item 
To estimate $\tr\{\bfS^{-1}\}$ we use $\left[1-\frac{p}{K}\right] \tr\{ {\hat\bfS}^{-1}\}.$
Result (\ref{eq:fifthmom}) shows this estimator is exactly unbiased (under the assumption $K>p).$
\item
Results (\ref{eq:thirdmom}) and (\ref{eq:fourthmom}) together show that 
\begin{equation}\label{eq:bigvar}
\left[1-\frac{p}{K}\right]^2 \tr\{ {\hat\bfS}^{-2}\}-\frac{1}{K} 
\left[1-\frac{p}{K}\right]\tr^2\{ {\hat\bfS}^{-1}\}
\end{equation}
is unbiased for $\tr\{ {\bfS}^{-2}\}$ with $K>p+1.$ However this estimator was found by simulation to have a very high variance, with occasional negative values. Better results were obtained by estimating $\tr\{ {\bfS}^{-2}\}$ using just the first term of  (\ref{eq:bigvar}), namely
$\left[1-\frac{p}{K}\right]^2 \tr\{ {\hat\bfS}^{-2}\},$ which is only asymptotically unbiased.
\end{itemize}
\smallskip\par
These estimators of the trace terms were used in the three estimators, now called QLa-est, QLb-est and QLP-est for this full estimation situation.
Repeating the simulations in Section~\ref{subsec:pcohsim}\/ 
for just the three estimators
gave the results in Fig.~\ref{fig:PRISE_EST} and lines 6 to 8 of Table~\ref{tab:avPRISE}. For individual one QLb-est performs best but at frequencies where the true partial coherence is high (2 and 3.25 Hz) it does poorly. It is generally quite variable.
QLP-est seems to deliver very significant improvements over the raw estimates, and whilst generally slightly inferior to QLb-est, is relatively frequency invariant and behaves better for  high partial coherencies. 
For individual two QLb-est slightly outperforms QLP-est, but both do very well in this sparse case. QLa-est is generally inferior to the other two but appears relatively unaffected by spikes of high true partial coherence.

In closing, we note that (\ref{eq:PalpQL}) and (\ref{eq:defC0}) involve only $\bfS$ and $\bfS^2$ which are more accurately estimated than their inverses
which appear in QLa-est and QLb-est; this undoubtedly adds greatly to the reliability and robustness of QLP-est. We also note that with $\beta_0$ in the form given in (\ref{eq:PalpQL}) that there is a $\tr\{\bfS\}$ in the denominator which is playing the role of $\tr\{ {\hat\bfS}\}$ in the previously accepted formulation of (\ref{eq:Haff}).
Hence 
we recommend QLP-est as our estimator of choice.

\section{Concluding Discussion}
We have considered how to improve estimation of partial coherencies from poorly-conditioned matrices. Our study carried out exact simulations derived from real EEG data for two individuals, one having large, and the other  small, partial coherencies. 
When true partial coherencies are quite large then for shrinkage estimators of the diagonal weighting kind  --- for spectral matrices or precision matrices --- minimization of risk using  QL leads to oracle partial coherence estimators  superior to HS equivalents. When true partial coherencies are small the methods behave similarly. For the full estimation (non-oracle) case,  QLP-est seems particularly robust and reliable and based  on the results here is our recommended approach.

\appendix{}
To simplify notation we drop  explicit frequency dependence. 
\subsection{Useful Expectations}\label{app:expect}
The first three of the following results are given in \cite{MaiwaldKraus00}. The third and fourth can be deduced from \cite[p.~308]{LetacMassam04}.
Let $\bfA, \bfB$ be arbitrary complex-valued matrices. Under (\ref{eq:eqnWish}),
\begin{eqnarray}
E\{ \tr\{ \bfA{\hat\bfS}\bfB{\hat\bfS} \} \} &=& \tr\{ \bfA\bfS\bfB\bfS \}\nonumber\\
&\,\,&+\textstyle{\frac{1}{K}} \tr\{ \bfA\bfS \}\tr\{ \bfB\bfS \}
\label{eq:firstmom}\\
E\{ \tr^2\{ {\hat\bfS} \} \} \!&=&\! \tr^2\{ \bfS \}+\textstyle{\frac{1}{K}} \tr\{ \bfS^2 \}
\label{eq:secondmom}\\
E\{ \tr\{ {\hat\bfS}^{-1} \} \} &=& \textstyle{\frac{K}{K-p}} \tr\{\bfS^{-1}\}\label{eq:fifthmom}\\
\!\!\!\!\!\!\!\!\!
E\{ \tr\{ \bfA{\hat\bfS}^{-1}\bfB {\hat\bfS}^{-1}\} \} &=& c_1\left[(K-p)\tr\{ \bfA\bfS^{-1}\bfB\bfS^{-1} \} \right.\nonumber \\
&\,\,&+\ \tr\{ \bfA{\bfS}^{-1}\}\tr\{ \bfB{\bfS}^{-1} \}]
\label{eq:thirdmom}\\
E\{ \tr^2\{ {\hat\bfS}^{-1} \} \} &=& c_1\left[(K-p)\tr^2\{ \bfS^{-1} \} \right. \nonumber\\
&\,\,&+\ \tr\{ {\bfS}^{-2} \}].
\label{eq:fourthmom}
\end{eqnarray}
Here (\ref{eq:firstmom}) and (\ref{eq:secondmom}) hold for $K>p-1,$ (\ref{eq:fifthmom}) holds for $K>p$ and (\ref{eq:thirdmom}) and (\ref{eq:fourthmom}) hold for $K>p+1.$ $c_1$ is given by
$$c_1=\frac{K^2}{(K-p)^3-(K-p)}.$$
\subsection{Derivative of trace of squared matrix}\label{app:diff}
If $x$ is a scalar, $\bfF(x)$ is an $m \times n$ matrix and $\bfG(x)$ is an $n\times q$ matrix, then \cite[p.~301]{Harville97}
$
{\partial (\bfF\bfG)}/{\partial x}=\bfF ({\partial \bfG}/{\partial x})+
 ({\partial \bfF}/{\partial x})\bfG.
$
Also, \cite[p.~304]{Harville97}
$
{\partial \,\tr\{\bfF\}}/{\partial x}=\tr\left\{ {\partial \bfF}/{\partial x} \right\}.
$
So,
\begin{eqnarray*}
\frac{\partial \,\tr\{\bfF\bfG\}}{\partial x} &=& 
\tr\left\{ 
\bfF \frac{\partial \bfG}{\partial x}+
 \frac{\partial \bfF}{\partial x}\bfG \right\}\\
&=& \tr\left\{ 
\bfF \frac{\partial \bfG}{\partial x}\right\}+
 \tr\left\{\bfG \frac{\partial \bfF}{\partial x}\right\}\\
&=& \tr\left\{ 
 \frac{\partial \bfG}{\partial x}\bfF\right\}+
 \tr\left\{ \frac{\partial \bfF}{\partial x}\bfG\right\},
\end{eqnarray*}
so if $\bfF=\bfG,$
\begin{equation}\label{eq:diffres}
\frac{\partial \,\tr\{\bfF^2\}}{\partial x}=
2\, \tr\left\{ 
 \frac{\partial \bfF}{\partial x}\bfF\right\}=2\,\tr\left\{ \bfF\frac{\partial \bfF}{\partial x}\right\}.
\end{equation}

\subsection{Proof of Lemma~\ref{lemma:QL2}}\label{app:rhoQLDet}
From (\ref{eq:riskQL}) the quadratic loss is  
${\mathcal R}_{QL}(\hat{\bfS}^\star(\rho), \bfS) = E\{\tr\{({\hat\bfS}^\star\bfS^{-1} - \bfI_p)^2\}\}.
$ 
So
$$\frac{\partial}{\partial \rho}{\mathcal R}_{QL}(\hat{\bfS}^\star(\rho), \bfS) = E\left\{\frac{\partial}{\partial \rho}\tr\{({\hat\bfS}^\star\bfS^{-1} - \bfI_p)^2\}\right\}.
$$ 
We now use (\ref{eq:diffres}) to find the derivative by setting 
\begin{eqnarray*}
\bfF={\hat\bfS}^\star\bfS^{-1} - \bfI_p=
[(1-\rho) {\hat{\bfS}}+\frac{\rho}{p}\tr\{\bfS\}\bfI_p] \bfS^{-1} - \bfI_p.
\end{eqnarray*}
Then,
\begin{eqnarray*}
\frac{\partial \bfF}{\partial \rho}\bfF &=&
\left(-{\hat\bfS}\bfS^{-1}+\frac{1}{p}\tr\{\bfS\}\bfS^{-1}\right)\\
&\times&\left([(1-\rho) {\hat{\bfS}}+\frac{\rho}{p}\tr\{\bfS\}\bfI_p] \bfS^{-1} - \bfI_p\right)\\
&=&-(1-\rho){\hat\bfS}\bfS^{-1}{\hat\bfS}\bfS^{-1}
-\frac{\rho}{p}\tr\{\bfS\}{\hat\bfS}\bfS^{-2}+{\hat\bfS}\bfS^{-1}\\
&+&\frac{(1-\rho)}{p}\tr\{\bfS\}\bfS^{-1}{\hat\bfS}\bfS^{-1}+\frac{\rho}{p^2}\tr^2\{\bfS\}\bfS^{-2}\\
&-&\frac{1}{p}\tr\{\bfS\}\bfS^{-1}.
\end{eqnarray*}
Using (\ref{eq:firstmom}),
\begin{eqnarray}
E\{   \tr\{\bfS^{-1}{\hat\bfS}\bfS^{-1}{\hat\bfS}\}\}&=&[ \tr\{\bfS^{-1}{\bfS}\bfS^{-1}{\bfS}\}+\frac{1}{K}\tr^2\{ \bfS^{-1}\bfS\}]\nonumber\\
&=& \left[p+\frac{p^2}{K}\right].
\label{eq:partone}
\end{eqnarray}
Using (\ref{eq:partone}) we obtain
\begin{eqnarray*}
\frac{\partial}{\partial \rho}{\mathcal R}_{QL}(\hat{\bfS}^\star(\rho), \bfS)\!\!\! &=&\!\!\!
E\left\{ 2\, \tr\left\{ 
 \frac{\partial \bfF}{\partial \rho}\bfF\right\}\right\}\\
&=&\!\!\!\!\!
[-2+2\rho]\!\left[p+\frac{p^2}{K}\right]-\frac{4\rho}{p}
\tr\{\bfS\}\tr\{\bfS^{-1}\}\\
&\quad&+2\frac{\rho}{p^2}\tr^2\{\bfS\}\tr\{\bfS^{-2}\}+2p.
\end{eqnarray*}
setting the result to zero and tidying gives
$$
\rho\left(\left[p+\frac{p^2}{K}\right]-\frac{2}{p}
\tr\{\bfS^{-1}\}\tr\{ \bfS\}+\frac{1}{p^2}\tr^2\{\bfS\}\tr\{\bfS^{-2}\}\right)=\frac{p^2}{K}
$$
which gives the expression (\ref{eq:rhoQLDet}) for $\rho.$ 

The second derivative is given by
\begin{eqnarray*}
\frac{\partial^2}{\partial \rho^2}{\mathcal R}_{QL}(\hat{\bfS}^\star(\rho), \bfS)\!\!\! 
&=&\!\!\!\!\!
2\left[p+\frac{p^2}{K}\right]-\frac{4}{p}
\tr\{\bfS\}\tr\{\bfS^{-1}\}\\
&\quad&+\frac{2}{p^2}\tr^2\{\bfS\}\tr\{\bfS^{-2}\}.
\end{eqnarray*}
We now write this as a quadratic in $\tr\{\bfS\},$ i.e., as
$
a\,\tr^2\{\bfS\}-b\,
\tr\{\bfS\}+c,
$
where 
$$a=\frac{2}{p^2}\tr\{\bfS^{-2}\};\quad
b=\frac{4}{p}\tr\{\bfS^{-1}\}; \quad
c=2\left[p+\frac{p^2}{K}\right].
$$
Next complete the square to obtain
$
\left(\sqrt a\, \tr\{\bfS\}- {b}/{[2\sqrt a]}\right)^2 -  {b^2}/{(4a)}+c.
$
This will be positive if $c- ({b^2}/{4a})$ is positive; i.e.,
$$
2\left[p+\frac{p^2}{K}\right]-2\frac{\tr^2\{\bfS^{-1}\}}{\tr\{\bfS^{-2}\}}>0.
$$
By Chebyshev's inequality we have
$
p\sum_{j=1}^p\lambda_j^{-2}\geq \left(\sum_{j=1}^p\lambda_j^{-1}\right)^2,
$
where the $\lambda_j$'s are eigenvalues of $\bfS.$ So 
$p \geq {\tr^2\{\bfS^{-1}\}}/{\tr\{\bfS^{-2}\}},
$
and so
$$
2\left[p+\frac{p^2}{K}\right]-2\frac{\tr^2\{\bfS^{-1}\}}{\tr\{\bfS^{-2}\}} \geq 2\frac{p^2}{K}>0,
$$
which proves that the turning point is a minimum.
\subsection{Proof of Lemma~\ref{lemma:twopar}}\label{app:twoparamorig}
For the HS loss we have that 
$({\partial}/{\partial \eta}){\mathcal R}_{HS}({\hat\bfS}^{\star}(\eta,\rho),\bfS)$ is given by $({\partial}/{\partial \eta})E\{\tr\{[(1-\rho){\hat\bfS}+\eta\rho \bfI_p - \bfS]^2\}\}$ which is 
\begin{eqnarray*}
E\{\tr\{2\rho\bfI_p[(1-\rho){\hat\bfS}+\eta\rho \bfI_p - \bfS]\}\}
\!\!&=&\!\! 2\rho\,\tr\{\eta\rho \bfI_p-\rho\bfS \}\nonumber\\
\!\!&=&\!\! 2\rho^2(\eta p -\tr\{\bfS\}),
\end{eqnarray*}
where we have used (\ref{eq:diffres}) in Appendix \ref{app:diff} and $E\{ {\hat\bfS} \}=\bfS.$
Then setting to zero gives
\begin{equation}
\eta_0 = \tr\{\bfS\} / p. \label{eq:alphaOptHS}
\end{equation}
Next,
$({\partial}/{\partial \rho}){\mathcal R}_{HS}(\hat{\bfS}^{\star}(\eta,\rho),\bfS)$ is given by
$({\partial}/{\partial \rho})E\{\tr\{[(1-\rho){\hat\bfS}+\eta\rho \bfI_p - \bfS]^2\}\}$ which is
\begin{eqnarray}
&&\!\!\!\!\!\!\!\!\!\!\!\!\!\!\!\!\!\!\!\!\!\!\!\!\!\!\!\!\!\!\!\!\!\!\!\!\!\!\!\!\!\!\!\!\!\!\!\!\!\!\!
E\{\tr\{2(\eta\bfI_p - {\hat\bfS})[(1-\rho){\hat\bfS}+\eta\rho \bfI_p - \bfS]\}\}\nonumber\\
&&\!\!\!\!\!\!\!\!\!\!\!\!\!\!\!\!\!\!\!\!\!\!\!\!\!\!\!\!\!\!\!\!\!\!\!\!\!\!\!= E\{\tr\{2(\eta\bfI_p - {\hat\bfS})[\rho(\eta\bfI_p - {\hat\bfS}) + {\hat\bfS} - \bfS]\}\}\nonumber\\
\Rightarrow
\rho_0 \DefEq \rho_0(\eta_0) &=& \frac{E\{\tr\{(\eta_0\bfI_p - {\hat\bfS})(\bfS- {\hat\bfS})\}\}}
{E\{\tr\{(\eta_0\bfI_p - {\hat\bfS})^2\}\}}. \label{eq:rhoOptHS}
\end{eqnarray}
The optimal shrinkage coefficient is dependent on the form of the target matrix, namely $\eta\bfI_p$. 
Look first at the numerator, $a$ say, of (\ref{eq:rhoOptHS}). Expanding gives
\begin{equation*}
a=E\{ \tr\{\eta_0\bfS-\eta_0{\hat\bfS} - {\hat\bfS}\bfS+ {\hat\bfS}^2\}\}.
\end{equation*}
Now $K\hbS$ has the complex Wishart distribution with mean $K\bfS.$
Since $E\{ {\hat\bfS}\}=\bfS$ the first two terms cancel.   Using (\ref{eq:secondmom}) we get
$$
a=-\tr\{ \bfS^2 \}+\tr\{ \bfS^2 \} + \frac{1}{K} \tr^2\{\bfS\}=\frac{1}{K} \tr^2\{\bfS\}.
$$
Expanding the denominator, $b$ say, of  (\ref{eq:rhoOptHS}) gives,
with $\eta_0=\tr\{ \bfS \}/p,$
\begin{eqnarray*}
b &=& E\{ \tr\{\eta_0^2\bfI_p - 2\eta_0{\hat\bfS}+ {\hat\bfS}^2 \}\}\\
&=& \frac{1}{p}\tr^2\{\bfS\}-\frac{2}{p}\tr^2\{\bfS\}+\tr\{ \bfS^2\}+ \frac{1}{K} \tr^2\{\bfS\}\\
&=& \left[ \frac{1}{K} -\frac{1}{p} \right]\tr^2\{\bfS\}+\tr\{ \bfS^2\}.
\end{eqnarray*}
The ratio $a/b$ then has the form (\ref{eq:HSrhoa}) or
(\ref{eq:theorHS}).

For the QL loss, we have that
$({\partial}/{\partial \eta}){\mathcal R}_{QL}({\hat\bfS}^{\star}(\eta,\rho),\bfS)$ is given by $({\partial}/{\partial \eta})E\{\tr\{ [(1-\rho){\hat\bfS}+\eta\rho\bfI_p]\bfS^{-1}-\bfI_p)^2\}\}$
which is
\begin{eqnarray}
&&\!\!\!\!\!\!\!\!\!\!\!\!\!\!\!\!\!\!\!\!\!\!\!\!\!\!\!\!\!\!\!\!\!\!\!\!\!\!\!\!\!\!E\{\tr\{2\rho\bfS^{-1}([(1-\rho){\hat\bfS}+\eta\rho \bfI_p ]\bfS^{-1}-\bfI_p)\}\}\nonumber\\
&=& -2\rho^2\tr\{\bfS^{-1}\}+2\rho^2\,\eta\, \tr\{\bfS^{-2}\},\nonumber\\
\Rightarrow \eta_0 &=& \tr\{\bfS^{-1}\}/\tr\{\bfS^{-2}\}. \label{eq:alphaOptQL}
\end{eqnarray}
Here we have again used (\ref{eq:diffres}) in Appendix \ref{app:diff} and $E\{ {\hat\bfS} \}=\bfS.$
This is very different in form to the scaling parameter (\ref{eq:alphaOptHS}) under HS loss.
Then $({\partial}/{\partial \rho}){\mathcal R}_{QL}({\hat\bfS}^{\star}(\eta,\rho),\bfS)$ is given by
$({\partial}/{\partial \rho})E\{\tr\{ [(1-\rho){\hat\bfS}+\eta\rho\bfI_p]\bfS^{-1}-\bfI_p)^2\}\}$
which is
\begin{eqnarray*}
&&\!\!\!\!\!\!E\{\tr\{2(\eta \bfI_p-{\hat\bfS})\bfS^{-1}([(1-\rho){\hat\bfS}+\eta\rho \bfI_p]\bfS^{-1}-\bfI_p)\}\}\\
&&=2 E\{ \tr\{\rho (\eta\bfI_p-{\hat\bfS})\bfS^{-1}(\eta\bfI_p-{\hat\bfS})\bfS^{-1}\\
&&+(\eta\bfI_p-{\hat\bfS})\bfS^{-1}({\hat\bfS}-\bfS)\bfS^{-1}\}\},\end{eqnarray*}
where we have again used (\ref{eq:diffres}) in Appendix \ref{app:diff}. Setting to zero gives
\begin{equation}
\rho_0\EqualDef \rho_0(\eta_0) = \frac{E\{\tr\{(\eta_0\bfI_p-{\hat\bfS})\bfS^{-1}
({\bfS-\hat\bfS})\bfS^{-1}\}\}
}{E\{\tr\{ (\eta_0\bfI_p-{\hat\bfS})\bfS^{-1}(\eta_0\bfI_p-{\hat\bfS})\bfS^{-1}\}\}}. \label{eq:rhoOptQL}
\end{equation}
Look first at the numerator, $a$ say, of (\ref{eq:rhoOptQL}). Expanding gives
\begin{equation*}
a=E\{ \tr\{\eta_0\bfS^{-1}-\eta_0\bfS^{-1}{\hat\bfS}\bfS^{-1} - {\hat\bfS}\bfS^{-1}+ {\hat\bfS}\bfS^{-1}{\hat\bfS}\bfS^{-1} \}\}.
\end{equation*}
Again, we use that $K\hbS$ has the complex Wishart distribution with mean $K\bfS.$
Since $E\{ {\hat\bfS}\}=\bfS$ the first two terms cancel. Then $E\{\tr\{{\hat\bfS}\bfS^{-1}\}\}=\tr\{\bfI_p\}=p.$  

From (\ref{eq:partone}) we know that 
$
E\{ \tr\{ {\hat\bfS}\bfS^{-1}{\hat\bfS}\bfS^{-1} \}\}=(p+(p^2/K)).
$
So  $a=-p+(p+(p^2/K))=p^2/K.$ 

Expanding the denominator, $b$ say, of  (\ref{eq:rhoOptQL}) gives
\begin{equation}\label{eq:proofden}
b=E\{ \tr\{\eta_0^2\bfS^{-2}-\eta_0\bfS^{-1}{\hat\bfS}\bfS^{-1} - \eta_0{\hat\bfS}\bfS^{-2}+ {\hat\bfS}\bfS^{-1}{\hat\bfS}\bfS^{-1} \}\}.
\end{equation}
Then
\begin{eqnarray*}
b&=&\frac{\tr^2\{\bfS^{-1}\}}{\tr^2\{\bfS^{-2}\}} \tr\{\bfS^{-2}\} -
2\frac{\tr^2\{\bfS^{-1}\}}{\tr\{\bfS^{-2}\}} + (p+(p^2/K))\\
&=& -\frac{\tr^2\{\bfS^{-1}\}}{\tr\{\bfS^{-2}\}} + (p+(p^2/K)).
\end{eqnarray*}
On tidying up the ratio $a/b$ is of the form (\ref{eq:theorQL}). That $\eta_0, \rho_0$ thus defined correspond to a minimum point is more easily shown via the proof of Lemma~\ref{lemma:reparam}.

\subsection{Proof of Lemma~\ref{lemma:reparam}}\label{app:twoparamnew}
Proceeding as before, and 
using (\ref{eq:secondmom}), 
for the HS loss
\begin{eqnarray*}
({\partial}/{\partial \alpha}){\mathcal R}_{HS}({\hat\bfS}^{\star}(\alpha,\beta),\bfS)&=&2[\alpha C +
\beta\tr\{ \bfS\}-\tr\{ \bfS^2\}]\\
({\partial}/{\partial \beta}){\mathcal R}_{HS}({\hat\bfS}^{\star}(\alpha,\beta),\bfS)&=&2[\alpha \tr\{ \bfS\}+\beta p-\tr\{ \bfS\}],
\end{eqnarray*}
where $C=\tr\{\bfS^2 \}+(1/K) \tr^2\{\bfS\}.$ Setting to zero and solving for $\beta$ gives
$$
\beta_0=
\frac{ \tr^3\{ \bfS\} }
{Kp \,\tr\{ \bfS^2 \} + [p-K]\tr^2\{ \bfS \} } = \rho_0\eta_0.
$$
Similarly, for $\alpha,$
$$
\alpha_0= \frac
{ p\, \tr\{ \bfS^2\}-\tr^2\{ \bfS\} }
{p[\tr\{ \bfS^2\}+ \frac{1}{K}\tr^2\{ \bfS\}]-\tr^2\{ \bfS\} } = 1-\rho_0.
$$
For the determinant of the Hessian matrix we have
$$
4\left|
\begin{matrix}
C & \tr\{ \bfS\}\\
\tr\{ \bfS\} & p
\end{matrix}
\right|_{\alpha_0,\beta_0} = 4p\, \tr\{ \bfS^2 \}-4[1-\frac{p}{K}]\tr^2\{\bfS\}.
$$
This is positive if $p\, \tr\{ \bfS^2 \} > [1-\frac{p}{K}]\tr^2\{\bfS\},$ i.e., $p > [1-\frac{p}{K}]\tr^2\{\bfS\}/\tr\{ \bfS^2 \},$ since $\tr\{ \bfS^2 \} >0.$ But by Chebyshev's inequality we know that
$p \geq \tr^2\{\bfS\}/\tr\{ \bfS^2 \},$ and so we know that $p > [1-\frac{p}{K}]\tr^2\{\bfS\}/\tr\{ \bfS^2 \},$ as required. Furthermore, 
$$
({\partial}^2/{\partial \beta^2}){\mathcal R}_{HS}({\hat\bfS}^{\star}(\alpha,\beta),\bfS)\Big|_{\alpha_0,\beta_0}=2p >0,
$$
and so we can conclude that the solution is a minimum.
For the QL loss,
\begin{eqnarray*}
({\partial}/{\partial \alpha}){\mathcal R}_{QL}({\hat\bfS}^{\star}(\alpha,\beta),\bfS)&=&2[\alpha D +
\beta\tr\{ \bfS^{-1}\}-p]\\
({\partial}/{\partial \beta}){\mathcal R}_{QL}({\hat\bfS}^{\star}(\alpha,\beta),\bfS)&=&2[\alpha \tr\{ \bfS^{-1}\}+\beta \tr\{ \bfS^{-2} \}\\
&-&2\tr\{ \bfS^{-1}\}],
\end{eqnarray*}
where $D=p[1+(p/K)].$ For the determinant of the Hessian matrix,
$$
4\left|
\begin{matrix}
D & \tr\{ \bfS^{-1}\}\\
\tr\{ \bfS^{-1}\} & \tr\{ \bfS^{-2} \}
\end{matrix}
\right|_{\alpha_0,\beta_0} = 4D\, \tr\{ \bfS^{-2} \}-4\tr^2\{\bfS^{-1}\}.
$$
This is positive if $D > \tr^2\{\bfS^{-1}\}/\tr\{ \bfS^{-2} \}, $ since $\tr\{ \bfS^{-2} \}>0.$
We know from Chebyshev's inequality that $p \geq \tr^2\{\bfS^{-1}\}/\tr\{ \bfS^{-2} \}, $ so  $p[1+(p/K)] > \tr^2\{\bfS^{-1}\}/\tr\{ \bfS^{-2} \}.$ Also we note that
$$
({\partial}^2/{\partial \alpha^2}){\mathcal R}_{QL}({\hat\bfS}^{\star}(\alpha,\beta),\bfS)\Big|_{\alpha_0,\beta_0}=2D >0,
$$
and so we can conclude that the solution is a minimum.

\subsection{Proof of Lemma~\ref{lemma:precisionHS}}\label{app:PrecisHS}
With ${\hat\bfC}^{\star}(\alpha,\beta)$ defined in
(\ref{eq:genshrinkc}) and ${\mathcal R}_{HS}(\hat{\bfC}^\star(\alpha,\beta), \bfC)$  we take 
$$\bfF={\hat\bfC}^{\star}\bfS-\bfS^{-1}=\alpha\,{\hat\bfS}^{-1}\bfS+\beta\,\bfS- \bfS^{-1}.
$$
Using (\ref{eq:fifthmom}) and (\ref{eq:thirdmom}) we find that
$({\partial}/{\partial \alpha}){\mathcal R}_{HS}({\hat\bfC}^{\star}(\alpha,\beta),\bfC)$ is given by
\begin{eqnarray}
&&\!\!\!\!\!\!\!\!\!\!\!\!\!\!\!\!\!\!\!\!\!\!\!\!\!\!\!\!\!\!\!\!\!\!\!\!
2\Big[\alpha c_1[(K-p)\tr\{\bfS^{-2}\}+\tr^2\{ \bfS^{-1} \}]\nonumber\\
&+&\beta{\textstyle{\frac{K}{K-p}}}\tr\{\bfS^{-1}\}-{\textstyle{\frac{K}{K-p}}}\tr\{\bfS^{-2}\}\Big]. \label{eq:HSPalpha}
\end{eqnarray}
Using (\ref{eq:fifthmom}) we find
$({\partial}/{\partial \beta}){\mathcal R}_{HS}({\hat\bfC}^{\star}(\alpha,\beta),\bfC)$ is given by
\begin{equation}
2\left[ \alpha \textstyle{\frac{K}{K-p}} \tr\{\bfS^{-1}\}+\beta p-\tr\{\bfS^{-1} \}\right].
\label{eq:HSPbeta}
\end{equation}
Setting (\ref{eq:HSPalpha}) and (\ref{eq:HSPbeta}) to zero and solving the simultaneous equations gives $\alpha_0$ and $\beta_0$ as stated in the lemma.
The determinant of the Hessian (divided by 4) is
$$
\delta=
\left|
\begin{matrix}
c_1 [(K-p) \tr\{\bfS^{-2}\}+\tr^2\{ \bfS^{-1}\}] & \frac{K}{K-p}\tr\{ \bfS^{-1}\}\\
\frac{K}{K-p}\tr\{ \bfS^{-1}\} & p 
\end{matrix}
\right|_{\alpha_0,\beta_0},
$$
which is 
$$\delta=
c_1 p[(K-p) \tr\{\bfS^{-2}\}+\tr^2\{ \bfS^{-1}\}]- \textstyle{\left[\frac{K}{K-p}\right]^2}\tr^2\{ \bfS^{-1}\}.
$$

Now $p\, \tr\{ \bfS^{-2} \} \geq \tr^2\{\bfS^{-1}\},$  so
\begin{eqnarray}
\delta &\!\!\!\!\!\!\!\!\!\geq&\!\!\!\!\!\!\!\!\! c_1 (K-p) \tr^2\{\bfS^{-1}\} + c_1p \tr^2\{\bfS^{-1}\}\!-\!\textstyle{\left[\frac{K}{K-p}\right]^2}\tr^2\{ \bfS^{-1}\}\nonumber\\
&=&\left(c_1 K - \textstyle{\left[\frac{K}{K-p}\right]^2}\right)\tr^2\{ \bfS^{-1}\}\nonumber\\
&=& \left(
\frac{pK^2(K-p)+K^2}{(K-p)^4-(K-p)^2}\right)
\tr^2\{ \bfS^{-1}\}>0,\nonumber
\end{eqnarray}
since $K>p+1.$ Combining the positive determinant with the fact that 
$({\partial^2}/{\partial \beta^2}){\mathcal R}_{HS}({\hat\bfC}^{\star}(\alpha,\beta),\bfC)=2p >0$ we see that the turning point is indeed a minimum. 
\subsection{Proof of Lemma~\ref{lemma:precisionQL}}\label{app:PrecisQL}
With ${\hat\bfC}^{\star}(\alpha,\beta)$ defined in
(\ref{eq:genshrinkc}) and ${\mathcal R}_{QL}(\hat{\bfC}^\star(\alpha,\beta), \bfC)$  we take 
$$\bfF={\hat\bfC}^{\star}\bfS-\bfI_p=\alpha\,{\hat\bfS}^{-1}\bfS+\beta\,\bfS- \bfI_p.
$$
Proceeding as before we find that
$({\partial}/{\partial \alpha}){\mathcal R}_{QL}({\hat\bfC}^{\star}(\alpha,\beta),\bfC)$ is given by
$$
2[\alpha 
E\{ \tr\{ {\hat\bfS}^{-1}\bfS {\hat\bfS}^{-1}\bfS \}\}\!+\! \beta E\{ \tr\{ {\hat\bfS}^{-1}\bfS^2\}\}\!-\!E\{\tr\{ {\hat\bfS}^{-1}\bfS \}\}].
$$ 
Setting $\bfA=\bfB=\bfS$ in  (\ref{eq:thirdmom}),
$$
E\{ \tr\{ {\hat\bfS}^{-1}\bfS {\hat\bfS}^{-1}\bfS \}\}=\frac{K}{K-p}{\frac{pK^2}{(K-p)^2-1}}=\frac{K}{K-p}c,
$$
where $c=pK^2/[(K-p)^2-1].$ 
Next we use the result that $E\{ {\hat\bfS}^{-1} \}=(K/(K-p))\bfS^{-1},$ which means that
$$
E\{ \tr\{ {\hat\bfS}^{-1}\bfS^2\}\}\!=\!\frac{K}{K-p}\tr\{ \bfS\}\,\mbox{and}\, E\{\tr\{ {\hat\bfS}^{-1}\bfS \}\}
\!=\! \frac{Kp}{K-p}.
$$
Putting results together we see that
$$
\frac{\partial}{\partial \alpha}{\mathcal R}_{QL}({\hat\bfC}^{\star}(\alpha,\beta),\bfC)=2 \frac{K}{K-p}\left[\alpha c+\beta \tr\{ \bfS\}-p\right].
$$
Next, $({\partial}/{\partial \beta}){\mathcal R}_{QL}({\hat\bfC}^{\star}(\alpha,\beta),\bfC)$ is given by
\begin{eqnarray*}
&&2[\alpha E\{\tr\{ \bfS {\hat\bfS}^{-1}\bfS\}\}+\beta E\{ \tr\{ \bfS^2 \}\}-E\{\tr\{ \bfS\}\}]\\
&&=2[ \alpha \frac{K}{K-p} \tr\{ \bfS\} +\beta 
\tr\{ \bfS^2 \}-\tr\{ \bfS\} ].
\end{eqnarray*}
Setting both partial derivatives to zero and removing constant multipliers gives the simultaneous equations
\begin{eqnarray}
\alpha c+\beta \tr\{ \bfS\}-p &=& 0 \label{eq:firstsim}\\
\alpha \frac{K}{K-p} \tr\{ \bfS\} +\beta 
\tr\{ \bfS^2 \}-\tr\{ \bfS\} &=& 0 \label{eq:secondsim}
\end{eqnarray}
Multiply (\ref{eq:firstsim}) by $K\tr\{\bfS\}/(K-p)$ and (\ref{eq:secondsim}) by $c.$ Subtracting, with $c=c_0 p,$ gives $\beta_0$ in 
(\ref{eq:PalpQL}) and (\ref{eq:defC0}).
Multiply (\ref{eq:firstsim}) by $\tr\{\bfS^2\}$ and (\ref{eq:secondsim}) by $\tr\{\bfS\}.$ Subtracting,  gives $\alpha_0$ in (\ref{eq:PalpQL}) and (\ref{eq:defC0}).

The determinant of the Hessian (divided by 4) is
$$
\left|
\begin{matrix}
\frac{K}{K-p}c & \frac{K}{K-p}\tr\{ \bfS\}\\
\frac{K}{K-p}\tr\{ \bfS\} & \tr\{ \bfS^{2} \}
\end{matrix}
\right|_{\alpha_0,\beta_0} \!\!\!= \textstyle{\frac{K}{K-p}}[c 
\tr\{ \bfS^{2} \}-\textstyle{\frac{K}{K-p}}
\tr^2\{\bfS\}].
$$
The term in the square bracket on the right-side can be rewritten as
\begin{equation}\label{eq:Hessproof}
\frac{pK^2}{(K-p)^2-1}\tr\{\bfS^2\}-\frac{K}{K-p}\tr^2\{\bfS\}
\end{equation}
We know that $p \tr\{\bfS^2\} \geq \tr^2\{\bfS\},$ so (\ref{eq:Hessproof}) will be positive if
$K^2/[(K-p)^2-1]- K/(K-p)>0.$ This term can be written as $[Kp(K-p)+K]/[(K-p)^3-(K-p)]>0$ since we are assuming $K>p+1.$ Combining the positive determinant with the fact that 
$({\partial^2}/{\partial \beta^2}){\mathcal R}_{QL}({\hat\bfC}^{\star}(\alpha,\beta),\bfC)=2\tr\{ \bfS^2 \} >0$ we see that the turning point is indeed a minimum. 
\section*{Acknowledgment}
The work of Deborah Schneider-Luftman was supported by EPSRC (UK). 

%
%
%
%
%
%
%
%
%
%
%
%
%
%
%
%
%
%
%
%
%
%
%
%
%
%
%
%
%
%
%
%
%
%
%

\end{document}